\documentclass[11pt,a4paper]{article}

\usepackage{amsfonts}
\usepackage{amsthm}
\usepackage{amsmath}
\usepackage{amssymb}
\usepackage{amscd}
\usepackage{mathrsfs}
\usepackage[mathscr]{eucal}
\usepackage{graphicx}
\usepackage{epstopdf} 
\usepackage{pict2e}
\usepackage{epic}
\usepackage{xcolor}
\usepackage{float}
\usepackage{mathtools}
\usepackage{centernot}

\usepackage{cite}
\usepackage{hyperref}
\hypersetup{colorlinks=true, urlcolor= black, linkcolor=black, citecolor=black}

\usepackage[utf8]{inputenc}
\usepackage[T1]{fontenc}
\usepackage{lmodern}
\usepackage{dsfont}
\usepackage{indentfirst}

\numberwithin{equation}{section}

\theoremstyle{plain}
\newtheorem{Thm}{Theorem}[section]
\newtheorem{Lem}[Thm]{Lemma}

\newtheorem{Prop}[Thm]{Proposition}

\theoremstyle{definition}
\newtheorem{Def}[Thm]{Definition}

\newtheorem{Rem}[Thm]{Remark}

\newtheorem*{Acknowledgements}{Acknowledgements}
\newtheorem*{Conflict}{Conflict of interest}

\usepackage[margin=3cm]{geometry}

\newcommand\n{\mathbf{n}}

\newcommand\sn{\partial\mathbf{n}}

\title{New lower bounds for the (near) critical Ising and $\varphi^4$ models' two-point functions}
\begin{document}

\author{Hugo Duminil-Copin\footnote{Institut des Hautes \'Etudes Scientifiques, \url{duminil@ihes.fr}}\ \footnotemark[2]\footnote{Université de Genève, \url{hugo.duminil@unige.ch}, \url{romain.panis@unige.ch}} , \, Romain Panis$^{\dag}$\footnote{Institut Camille Jordan, \url{panis@math.univ-lyon1.fr}}\footnotemark[2]}
\maketitle

\begin{abstract} We study the nearest-neighbour Ising and $\varphi^4$ models on $\mathbb Z^d$ with $d\geq 3$ and obtain new lower bounds on their two-point functions at (and near) criticality. Together with the classical infrared bound, these bounds turn into up to constant estimates when $d\geq 5$. When $d=4$, we obtain an ``almost'' sharp lower bound corrected by a logarithmic factor. As a consequence of these results, we show that $\eta=0$ and $\nu=1/2$ when $d\geq 4$, where $\eta$ is the critical exponent associated with the decay of the model's two-point function at criticality and $\nu$ is the critical exponent of the correlation length $\xi(\beta)$. When $d=3$, we improve previous results and obtain that $\eta\leq 1/2$. As a byproduct of our proofs, we also derive the blow-up at criticality of the so-called bubble diagram when $d=3,4$.
\end{abstract}

\section{Introduction}

 We are interested in two ferromagnetic and real-valued spin models on $\mathbb Z^d$ that are amongst the most studied in statistical mechanics: the Ising model and the 
(discrete) $\varphi^4$ model. These models are formally defined as follows: given $\Lambda\subset \mathbb Z^d$ finite, $\mathcal{S}\subset \mathbb R$, a probability measure $\rho$ on $\mathcal S$, and an inverse temperature $\beta\geq 0$, we define a probability measure on $\mathcal{S}^\Lambda$ according to the formula
\begin{equation}
    \langle F(\tau)\rangle_{\Lambda,\rho,\beta}:=\frac{1}{\mathbf{Z}_{\Lambda,\rho,\beta}}\int_{\mathcal{S}^\Lambda}F(\tau)\exp\Big(\beta\sum_{\substack{x,y\in \Lambda\\ x\sim y}}\tau_x\tau_y\Big)\prod_{x\in \Lambda}\mathrm{d}\rho(\tau_x),
\end{equation}
where $F:\mathcal S^\Lambda\rightarrow \mathbb R$, $\mathbf{Z}_{\Lambda,\rho,\beta}$ is the \emph{partition function} of the model ensuring that $\langle 1\rangle_{\Lambda,\rho,\beta}=1$, and $x\sim y$ means that $|x-y|_2=1$ (where $|\cdot|_2$ is the $\ell^2$ norm on $\mathbb R^d$). The Ising model corresponds to choosing $\mathcal{S}=\lbrace -1,1\rbrace$ and $\rho(\mathrm{d}x)=\tfrac{1}{2}(\delta_{-1}+\delta_1)$ above; while the $\varphi^4$ model corresponds to choosing $\mathcal{S}=\mathbb R$ and $\rho=\rho_{g,a}$ where $g>0$, $a\in \mathbb R$, and 
\begin{equation}
    \rho(\mathrm{d}t)=\rho_{g,a}(\mathrm{d}t)=\frac{e^{-gt^4-at^2}\mathrm{d}t}{\int_{\mathbb R }e^{-gs^4-as^2}\mathrm{d}s}.
\end{equation}
 In the case of the Ising model (resp.~the $\varphi^4$ model), we write $\sigma$ (resp.~$\varphi$) instead of $\tau$ for the spin (resp.~field) variable. We drop the subscript $\rho$ in the above notations.

It is a well-known fact (see \cite{griffiths1967correlations}) that the measures $\langle \cdot\rangle_{\Lambda,\beta}$ admit a weak limit when $\Lambda\nearrow\mathbb Z^d$. We denote it by $\langle \cdot \rangle_{\beta}$. When $d\geq 2$, the model undergoes a (second-order) phase transition (see \cite{aizenman1986critical,ADCS,gunaratnam2022random}) at a critical parameter $\beta_c=\beta_c(\rho)\in(0,\infty)$ that can be defined as follows (see \cite{aizenman1987phase})
\begin{equation}
	\beta_c:=\inf\Big\{\beta\geq 0: \: \chi(\beta):=\sum_{x\in \mathbb Z^d}\langle \tau_0\tau_x\rangle_\beta=\infty\Big\}.
\end{equation}

The connection between the Ising model and the $\varphi^4$ model is anticipated to be highly intricate, with both models believed to fall within the same universality class. Renormalisation group arguments, as detailed in \cite{G70, K93} and the recent book \cite{BBS19}, suggest that many of their properties, including \emph{critical exponents}, precisely coincide at their respective critical points. Griffiths and Simon \cite{simon1973varphi} laid the groundwork for these profound connections by demonstrating that the $\varphi^4$ model emerges as a specific near-critical scaling limit of a collection of mean-field Ising models. This allows for a rigorous transfer of numerous useful properties from the Ising model, such as correlation inequalities, to the $\varphi^4$ model. Conversely, the Ising model can be derived as a limit of the $\varphi^4$ model using the following (weak) convergence:
\begin{equation}
    \frac{\delta_{-1}+\delta_1}{2}=\lim_{g\rightarrow \infty}\frac{e^{-g(\varphi^2-1)^2+g}\textup{d}\varphi}{\int_{\mathbb R }e^{-gs^4+2gs^2}\mathrm{d}s}.
\end{equation}

The upper critical dimension of these models should be equal to $4$ meaning that they should exhibit \emph{mean-field behaviour} above dimension $d\geq 5$, and a \emph{marginal} (i.e.\ with potential logarithmic corrections) mean-field behaviour in dimension $d=4$. In such setups, the expected behaviour of the models undergoes a notable simplification, where lattice geometry ceases to exert a significant influence. A prominent approach to investigating the mean-field regime involves the computation of the critical exponents of the model. Despite substantial progress, determining these exponents systematically in the mean-field regime remains a challenging task. Noteworthy methods such as the \emph{lace expansion} and the \emph{renormalisation group method} have emerged as powerful alternatives, showcasing significant successes. In the realm of spin models on $\mathbb{Z}^d$, their success covers the weakly-coupled \cite{gawedzki1985massless,bauerschmidt2014scaling,slade2016critical,BBS19,michta2023boundary}, high-dimensional \cite{sakai2007lace,sakai2015application,sakai2022correct}, or sufficiently spread-out \cite{sakai2007lace,chen2015critical,chen2019critical,sakai2022correct} setups. Other milestones toward the study of the mean-field regime of these models include the proofs of triviality of the scaling limits at criticality in dimension $d\geq 5$ independently obtained by Aizenman \cite{A} and Fröhlich \cite{frohlich1982triviality}, together with the recent work of Aizenman and Duminil-Copin \cite{ADC} which establishes the corresponding result in dimension $d=4$ (see also \cite{Pan23+} for a treatment of models of \emph{effective} dimension $d_{\text{eff}}\geq 4$). 

The situation in dimension $d=3$ is more complicated and much less well understood. Obtaining a rigorous understanding of the critical $3d$ Ising model remains one of the main challenges of statistical mechanics. In the physics literature, recent progress has been made regarding this problem using the so-called \emph{conformal bootstrap}, see \cite{el2012solving,el2014solving,rychkov2017non}.

In this work, we study the critical exponent related to the decay of the model's two-point function at criticality when $d\geq 3$. It is expected that at criticality the two-point function of the above models decays algebraically, thus justifying the introduction of the critical exponent $\eta$ defined as follows: for all $x\in \mathbb Z^d\setminus\lbrace 0\rbrace$,
\begin{equation}
    \langle \tau_0\tau_x\rangle_{\beta_c} =\frac{1}{|x|^{d-2+\eta-o(1)}},
\end{equation}
where $|\cdot|$ denotes the infinite norm on $\mathbb R^d$ and $o(1)$ is a quantity tending to zero as $|x|$ tends to infinity.
Since these models are reflection positive (see \cite{FSS,FILS,B}), the \emph{infrared bound} and the Messager--Miracle-Solé (MMS) inequalities\footnote{To be more precise, it relies on the following simple consequences of the MMS inequalities which hold for all $\beta>0$ (see \cite{MMS} or e.g.~\cite[Proposition~5.1]{ADC}): $(i)$ the sequence $(\langle \tau_0\tau_{k\mathbf{e}_1}\rangle_\beta)_{k\geq 0}$ is decreasing; $(ii)$ for any $x\in \mathbb Z^d$, one has 
\begin{equation*}
	\langle \tau_0\tau_{(|x|_1,0_\bot)}\rangle_\beta
\leq\langle \tau_0\tau_x\rangle_\beta\leq \langle \tau_0\tau_{(|x|,0_\bot)}\rangle_\beta,
\end{equation*} where $|\cdot |_1$ denotes the $\ell^1$ norm on $\mathbb R^d$, and where $0_\bot\in \mathbb Z^{d-1}$ is null vector.} \cite{MMS} provide the existence of $C=C(d)>0$ such that for all $\beta\leq \beta_c$, and all $x\in \mathbb Z^d\setminus \lbrace 0\rbrace$,
\begin{equation}\label{eq: Infrared bound}
\langle \tau_0\tau_x\rangle_\beta\leq \frac{C}{|x|^{d-2}}.
\end{equation}
Note that this bound does not provide any interesting information when $d=2$.
Moreover, as a consequence of the Simon--Lieb inequality \cite{simon1980correlation,lieb2004refinement}, Simon proved the existence of $c=c(d)>0$ such that for all $x\in \mathbb Z^d\setminus\lbrace 0\rbrace$, 
\begin{equation}\label{eq: Lower bound naive}
    \langle \tau_0\tau_x\rangle_{\beta_c}\geq \frac{c}{|x|^{d-1}}.
\end{equation}
These two results show that (if it exists) the critical exponent $\eta$ satisfies $0\leq \eta \leq 1$. We now survey existing results regarding the computation of $\eta$.

In dimension $d=2$, the transfer matrix formalism leads to exact solutions of the Ising model \cite{onsager1944crystal,kaufman1949crystal,yang1952spontaneous}. This formalism was later used by Wu \cite{wu1966theory,mccoy1973two} to demonstrate that $\eta=1/4$. Much more information was obtained on the two-point function through the proof of conformal invariance of Smirnov \cite{smirnov2010conformal}, see for instance \cite{chelkak2015conformal} for a proof of conformal invariance of the spin correlations.

When $d=3$, the results are confined to the bounds $0\leq \eta\leq 1$. However, the conformal bootstrap method provides precise predictions \cite{kos2016precision} on the value of $\eta$, with the most accurate numerical estimate being $\eta\approx 0.0362978(20)$.

When $d\geq 4$, the mean-field regime can be derived. Lace expansion methods \cite{sakai2007lace,sakai2015application,sakai2022correct} were applied in very large dimensions to obtain, not only that for both models $\eta=0$, but also exact asymptotics for the critical two-point function, showing that it is equivalent (at large scales) to $A/|x|_2^{d-2}$ (where $A>0$ is a model-dependent constant). For the case of the weakly-coupled $\varphi^4$ model (i.e.~with small coupling $g$), lace expansion was successfully implemented for $d\geq 5$ in \cite{brydges2021continuous}. The renormalisation group method was applied up to dimension $d=4$ to obtain exact asymptotics in this setup. More precisely, it was shown in \cite{slade2016critical} that for the weakly-coupled $\varphi^4$ model in dimension $4$,
\begin{equation}\label{eq: exact asymptotic d=4 wc phi4}
	\langle \varphi_0\varphi_x\rangle_{\beta_c}=\frac{A}{|x|_2^2}(1+o(1)),
\end{equation}
where $A>0$ and $o(1)$ tends to $0$ as $|x|_2$ tends to infinity. Away from these perturbative regimes, the best bounds on $\eta$ are the ones given by \eqref{eq: Infrared bound} and \eqref{eq: Lower bound naive}.

In this paper, we obtain new (near) critical lower bounds on the two-point functions of these models when $d\geq 3$. In particular, we establish a sharp lower bound when $d\geq 5$, and an almost sharp lower bound (with a logarithmic error\footnote{Following the universality hypothesis, we expect the critical two-point functions of the four-dimensional models of interest to behave as in \eqref{eq: exact asymptotic d=4 wc phi4}.}) when $d=4$, showing that $\eta=0$ in these dimensions. We also obtain a new bound on $\eta$ when $d=3$: if it exists, $\eta\leq 1/2$.
When $d\geq 4$, these results allow us to compute the critical exponent $\nu$ associated with the blow-up near criticality of the correlation lengths $L(\beta),\: \xi(\beta)$, and $\xi_p(\beta)$ for $p>0$ (see \eqref{eq: def L beta}, \eqref{eq: def correlation length}, and \eqref{eq: def correlation length p} below). To the best of our knowledge, this result is new in our setup. Let us mention that the exponents $\eta$ and $\nu$ are usually more complicated to derive than the exponents $\alpha,\beta,\delta,\gamma$ (for the second derivative of the free energy, the magnetisation, and the susceptibility, see \cite{A,aizenman1983renormalized,aizenman1986critical,bauerschmidt2014scaling}) as they depend on the graph metric rather than intrinsic distances.
 Although only stated in the case of the Ising and the $\varphi^4$ models, we will show that these results extend for measures $\rho$ that belong to the Griffiths--Simon class of measures, see Section \ref{section: Gs class of measures}.

\subsection{The main theorem: a new inequality for two-point functions} Before stating the main results of this work, we introduce the distance below which, for $\beta<\beta_c$, the model recovers critical features.  The following distance, called the \emph{sharp length}, was introduced in \cite{DCT,Pan23+}.

\begin{Def}[Sharp length] Let $\rho$ correspond to either the Ising or the $\varphi^4$ model. Let $\beta>0$. Let $S$ be a finite subset of $\mathbb Z^d$ containing $0$ and set
\begin{equation}
    \varphi_{\rho,\beta}(S):=\beta\sum_{\substack{x\in S\\ y\notin S, \: y\sim x}}\langle \tau_0\tau_x\rangle_{S,\rho,\beta}.
\end{equation}
Define the sharp length by
\begin{equation}\label{eq: def L beta}
    L(\beta)=L(\rho,\beta):=\inf\left\lbrace k\geq 1: \:\exists S\subset \mathbb Z^d,\: 0\in S, \: \textup{diam}(S)\leq 2k, \: \varphi_{\rho,\beta}(S)< 1/2\right\rbrace,
\end{equation}
where $\textup{diam}(S):=\max \lbrace |x-y|, \: x,y\in S\rbrace$. Note that $L(\beta_c)=\infty$ (see \cite[Section~3.6]{Pan23+}).
\end{Def}
With this definition at hand, the Simon--Lieb inequality together with the infrared bound \eqref{eq: Infrared bound} yield the following result\footnote{For sake a completeness, here is a proof. Let $\beta\leq \beta_c$ and $S\subset \mathbb Z^d$, finite, containing $0$, of diameter smaller than $2L(\beta)$, and such that $\varphi_{\rho,\beta}(S)<\tfrac{1}{2}$. By \eqref{eq: Infrared bound}, we may restrict ourselves to the case $|x|>2L(\beta)$. Iterating the Simon--Lieb inequality $k:=\lfloor |x|/2L(\beta)\rfloor-1$ times with translates of $S$ gives
\begin{equation*}
	\langle \tau_0\tau_x\rangle_{\beta}\leq \varphi_{\rho,\beta}(S)^k\max\{\langle \sigma_y\sigma_x\rangle_\beta : y\notin \Lambda_{L(\beta)}(x)\}\leq 2^{-k}\frac{C}{L(\beta)^{d-2}},
\end{equation*}
where we used \eqref{eq: Infrared bound} in the second inequality. Equation \eqref{eq: near-critical upper bound} follows from choosing $c$ appropriately.}: there exist $c,C>0$ such that, if $\beta\leq \beta_c$ and $x\in \mathbb Z^d\setminus \lbrace 0\rbrace$,
\begin{equation}\label{eq: near-critical upper bound}
    \langle \tau_0\tau_x\rangle_\beta\leq C\Big(\frac{1}{|x|\wedge L(\beta)}\Big)^{d-2}\exp\Big(-c\frac{|x|}{L(\beta)}\Big).
\end{equation}

Our main result is the following inequality. If $n\geq 1$, introduce the hyperplane $\mathbb H_n:=\lbrace x\in \mathbb Z^d: \: x_1=n\rbrace$ and denote by $\mathcal{R}_n$ the orthogonal reflection with respect to $\mathbb H_n$. Also introduce the box $\Lambda_n:=[-n,n]^d\cap \mathbb Z^d$.
\begin{Thm}\label{prop: intermediate step lower bound} Let $d\geq 3$. There exist $c_0,N_0>0$ such that for all $\beta\leq \beta_c$ and for all $N_0\leq n \leq L(\beta)$,
\begin{equation}\label{eq: prop intermediate step lower bound}
    \beta\sum_{\substack{x,y\in \Lambda_n\\y\sim x}}\left(\langle \tau_0\tau_x\rangle_{\beta}-\langle \tau_0\tau_{\mathcal{R}_n(x)}\rangle_{\beta}\right)\langle \tau_y \tau_{\mathcal{R}_n(y)}\rangle_\beta\geq c_0.
\end{equation}
\end{Thm}

We briefly explain the strategy of proof of the above inequality and stress that it does not use reflection positivity. We prove the result for the Ising model first and then extend it to the $\varphi^4$ model (and in fact, to all models in the Griffiths--Simon class of measures) using the Griffiths--Simon approximation (see Proposition \ref{prop: phi4 is GS}). Focusing now on the Ising model, one may obtain a lower bound on the model's two-point function at criticality by noticing that $\varphi_{\beta_c}(\Lambda_n)\geq 1$, and additionally using Griffiths' inequality together with the MMS inequalities (see for instance \cite[Section~3.6]{Pan23+}). Our method is a refinement of this inequality which consists in using the above observation for a well-chosen--- random--- set $S$ instead of $\Lambda_n$. The randomness of this set is obtained through the use of the \emph{random current representation} of the Ising model, see Section \ref{Section: rcr}. The result will follow from an appropriate application of the switching principle to a reflected current, see Section \ref{section: proof of the prop for the ising model}. Such reflected currents have been used in \cite{aizenman2019emergent} to derive a  Messager--Miracle-Solé inequality without using reflection positivity.

The regularity properties provided by reflection positivity (see \cite[Section~5]{ADC} or \cite[Section~3]{Pan23+}) allow one to turn the inequality of Theorem \ref{prop: intermediate step lower bound} into a pointwise lower bound on the two-point function that is valid for all dimensions $d\geq 3$. Let $\mathbf{e}_i$ be the unit vector of $i$-th coordinate equal to $1$.
\begin{Thm}\label{thm: pointwise lower bound all dimensions} Let $d\geq 3$. There exist $c_1,N_1>0$ such that for all $\beta\leq \beta_c$ and for all $N_1\leq n\leq L(\beta)$,
\begin{equation}\label{eq: lwer bound always true}
	\langle \tau_0\tau_{n\mathbf{e}_1}\rangle_{\beta}\geq \frac{c_1/\beta}{\displaystyle \chi_{4n}(\beta)+n^{d-2}\sum_{0\leq k \leq 4n} (k+2)\langle \tau_0\tau_{k\mathbf{e}_1}\rangle_\beta}.
\end{equation}
\end{Thm}
\begin{proof} Let $c_0,N_0>0$ be given by Theorem \ref{prop: intermediate step lower bound}. Let $\beta\leq \beta_c$ and $n\geq N_0$. The constants $C_i>0$ below only depend on the dimension. Divide the sum on the left-hand side of \eqref{eq: prop intermediate step lower bound} according to whether $-n\leq x_1\leq \lfloor n/2\rfloor $ or  $\lfloor n/2\rfloor  < x_1\leq n$.  By the MMS inequalities, 
\begin{align}
    \sum_{\substack{x,y\in \Lambda_n\\x_1\leq \lfloor n/2\rfloor \\y\sim x}}\left(\langle \tau_0\tau_x\rangle_{\beta}-\langle \tau_0\tau_{\mathcal{R}_n(x)}\rangle_{\beta}\right)\langle \tau_y \tau_{\mathcal{R}_n(y)}\rangle_\beta &\leq\sum_{\substack{x,y\in \Lambda_n\\ \: x_1\leq \lfloor n/2\rfloor \\y\sim x}}\langle\tau_0\tau_x\rangle_\beta\langle \tau_y\tau_{\mathcal{R}_n(y)}\rangle_\beta
    \\
    &\leq 
    2d\chi_n(\beta)\langle \tau_0\tau_{(n-2) \mathbf{e}_1}\rangle_\beta
    \\
    &\leq 2d\chi_n(\beta)\langle \tau_0\tau_{\lfloor n/4\rfloor}\rangle_\beta\label{eq: proof coro1},
 \end{align}
where we used that if $y$ contributes to the above sum, $|y-\mathcal R_n(y)|\geq 2[n-(\tfrac{n}{2}+1)]=n-2$.
If $\lfloor n/2\rfloor <x_1\leq n$, using the spectral representation of these models and the MMS inequalities, we obtain the following gradient estimate\footnote{The gradient estimate is obtained similarly as in \cite[Proposition~5.9]{ADC}: using the spectral representation of the Ising model, for all $x_\bot\in \mathbb Z^{d-1}$, there exists a measure $\mu=\mu(x_\bot,\beta)$ on $[0,1]$ such that for all $k\geq 1$,
\begin{equation*}
    \langle \tau_0\tau_{(k,x_\bot)}\rangle_{\beta} =
    \int_0^1 \lambda^k\text{d}\mu(\lambda).
\end{equation*}
The above display implies that
\begin{equation*}
    \langle \tau_0\tau_{(k,x_\bot)}\rangle_{\beta}-\langle\tau_0\tau_{(k+1,x_\bot)}\rangle_{\beta}
    =
    \dfrac{1}{k}\int_0^1 k\lambda^k(1-\lambda)\text{d}\mu(\lambda).
\end{equation*}
The estimate \eqref{eq: gradient estimate} follows by telescoping and replacing $k\lambda^{k}(1-\lambda)$ by $C\lambda^{\lfloor k/2\rfloor }$ for some constant $C>0$.}
\begin{equation}\label{eq: gradient estimate}
    \langle \tau_0\tau_x\rangle_\beta-\langle \tau_0\tau_{\mathcal{R}_n(x)}\rangle_\beta
    \leq 
    C_1\frac{|x-\mathcal{R}_n(x)|}{n}\langle \tau_0\tau_{\lfloor n/4\rfloor \mathbf{e}_1}\rangle_\beta.
\end{equation}
Using \eqref{eq: gradient estimate} and the MMS inequalities one more time, we get
\begin{align}
    \sum_{\substack{x,y\in \Lambda_n\\ x_1> \lfloor n/2\rfloor \\y\sim x}}\left(\langle \tau_0\tau_x\rangle_\beta-\langle \tau_0\tau_{\mathcal{R}_n(x)}\rangle_\beta\right)\langle \tau_y\tau_{\mathcal{R}_n(y)}\rangle_\beta
    &\leq C_1
    \frac{\langle \tau_0\tau_{\lfloor n/4\rfloor \mathbf{e}_1}\rangle_\beta}{n}\sum_{\substack{x,y\in \Lambda_n\\ \:x_1> \lfloor n/2\rfloor \\y\sim x}}|x-\mathcal R_n(x)|\langle \tau_y\tau_{\mathcal R_n(y)}\rangle_\beta \notag\\     &\leq C_2 
    \langle \tau_0\tau_{\lfloor n/4\rfloor \mathbf{e}_1}\rangle_\beta n^{d-2}\sum_{0\leq k\leq n}(k+2)\langle \tau_0\tau_{k\mathbf{e_1}}\rangle_{\beta}.\label{eq: eq: proof main theorem 1}
    \end{align}
Plugging \eqref{eq: proof coro1} and \eqref{eq: eq: proof main theorem 1} in \eqref{eq: prop intermediate step lower bound} yields
\begin{equation}
	\frac{c_1}{\beta}\leq C_3 \langle \tau_0\tau_{\lfloor n/4\rfloor\mathbf{e}_1}\rangle_\beta\Big(\chi_n(\beta)+n^{d-2}\sum_{0\leq k \leq n}(k+2)\langle \tau_0\tau_{k\mathbf e_1}\rangle_\beta\Big),
\end{equation}
from which the proof follows readily.
\end{proof}

\subsection{Applications}

We now list a number of applications and we include their (short) derivation. 

\paragraph{Lower bounds for the spin-spin correlations.} To begin, the infrared bound \eqref{eq: Infrared bound} allows to obtain a more explicit formulation of Theorem \ref{thm: pointwise lower bound all dimensions}. It can be interpreted as the fact that $\eta=0$ when $d\ge4$.
\begin{Thm}[Pointwise lower bound in dimension $d\ge4$]\label{thm: lower bound} Let $d\geq 4$.~There exists $c=c(d)>0$ such that for all $\beta\leq \beta_c$ and for all $x \in \mathbb Z^d$ with $2\leq |x|\leq L(\beta)$,
\begin{align}
    \langle \tau_0\tau_x\rangle_{\beta}\geq\begin{cases}\displaystyle\frac{c}{|x|^{d-2}} &\text{ if }d\ge5,\\
    \displaystyle\frac{c}{|x|^{2}\log |x|}&\text{ if }d=4.\end{cases}
\end{align}
\end{Thm}
\begin{proof} Let $c_1,N_1>0$ be given by Theorem \ref{thm: pointwise lower bound all dimensions}. Let $\beta\leq \beta_c$ and $N_1\leq n \leq L(\beta)$. Using the infrared bound \eqref{eq: Infrared bound} gives
\begin{equation}
	\chi_{4n}(\beta)+n^{d-2}\sum_{0\leq k \leq 4n} (k+2)\langle \tau_0\tau_{k\mathbf{e}_1}\rangle_\beta\leq C_1\Big(n^2+n^{d-2}\sum_{k=1}^{4n}\frac{1}{k^{d-3}}\Big)\leq \begin{cases}\displaystyle C_2n^{d-2} &\text{ if }d\ge5,\\
    \displaystyle C_3 n^2 \log n &\text{ if }d=4.\end{cases}
\end{equation}
 The proof follows readily by Theorem \ref{thm: pointwise lower bound all dimensions}, the MMS inequalities, and by choosing $c>0$ small enough (in particular to include the case $|x|\leq N_1$).
\end{proof}

When $d=3$ we do not obtain a more explicit pointwise lower bound, but we still improve on the existing bound on $\eta$.

\begin{Thm} Let $d=3$. If the critical exponent $\eta$ exists, it satisfies $\eta\leq \frac{1}{2}$.
\end{Thm}

\begin{proof} The proof follows by plugging the estimate provided by the existence of $\eta$ in \eqref{eq: lwer bound always true}. More precisely, the existence of $\eta$ implies that
\begin{equation}
	\langle \tau_0\tau_{k\mathbf{e}_1}\rangle_{\beta_c}=\frac{1}{k^{1+\eta+o(1)}},
\end{equation}
where $o(1)$ tends to $0$ as $k$ tends to infinity. Recall from \eqref{eq: Infrared bound} and \eqref{eq: Lower bound naive} that $\eta\in[0,1]$. Hence, one has
\begin{equation}\label{eq:precise proof bound eta d=3}
	\chi_{4n}(\beta_c)=n^{2-\eta+o(1)}, \qquad n\sum_{0\leq k\leq 4n}(k+2)\langle \tau_0\tau_{k\mathbf{e}_1}\rangle_{\beta_c}=n^{2-\eta+o(1)},
\end{equation}
so that, plugging the two previous displayed equations in Theorem \ref{thm: pointwise lower bound all dimensions},
\begin{equation}
	\frac{1}{n^{1+\eta+o(1)}}\geq \frac{1}{n^{2-\eta+o(1)}},
\end{equation}
which implies that $n^{1-2\eta+o(1)}\geq 1$, and thus that $\eta\leq \tfrac{1}{2}$.
\end{proof}

\paragraph{Behaviour of the correlation lengths.} Besides the sharp length $L(\beta)$, there are other natural typical lengths that one may define. The following quantity, called the \emph{correlation length}, is well defined for $\beta<\beta_c$
\begin{equation}\label{eq: def correlation length}
	 \xi(\beta):=-\lim_{n\rightarrow \infty}\left(\frac{1}{n}\log \langle \tau_0\tau_{n\mathbf{e}_1}\rangle_\beta\right)^{-1}.
\end{equation}
For $p>0$, one may also define
\begin{equation}\label{eq: def correlation length p}
	\xi_p(\beta):=\Big(\frac{1}{\chi(\beta)}\sum_{x\in \mathbb Z^d}|x|^p\langle \tau_0\tau_x\rangle_\beta\Big)^{1/p}.
\end{equation}
It is expected that there exist $\tilde\nu,\nu$, and $\nu_p>0$ such that for $\beta<\beta_c$,
\begin{equation}
	L(\beta)=(\beta_c-\beta)^{-\tilde\nu+o(1)}, \qquad \xi(\beta)=(\beta_c-\beta)^{-\nu+o(1)}, \qquad \xi_p(\beta)=(\beta_c-\beta)^{-\nu_p+o(1)},
\end{equation}
where $o(1)$ tends to zero as $\beta$ tends to $\beta_c$. The relation between these quantities is not clear a priori. Theorem \ref{thm: lower bound} allows to compute these exponents when $d\geq 4$.
\begin{Thm} Let $d\geq 4$. Let $p>0$. Then, for $\beta<\beta_c$,
\begin{align}\label{eq: correlation length 1}
	L(\beta)&= (\beta_c-\beta)^{-1/2+o(1)},
\\\label{eq: correlation length 2}
	\xi(\beta)&= (\beta_c-\beta)^{-1/2+o(1)},
\\\label{eq: correlation length 3}
	\xi_p(\beta)&=(\beta_c-\beta)^{-1/2+o(1)},
\end{align}
where $o(1)$ tends to $0$ as $\beta$ tends to $\beta_c$.
\end{Thm}
\begin{proof} By \cite{A,aizenman1983renormalized} we know that $\chi(\beta)=(\beta_c-\beta)^{-1+o(1)}$. Theorem \ref{thm: lower bound} and \eqref{eq: near-critical upper bound} yield that $\chi(\beta)=L(\beta)^{2+o(1)}$. Combined with the estimate on $\chi(\beta)$ this yields \eqref{eq: correlation length 1}. The estimate \eqref{eq: correlation length 3} follows by similar arguments. 
We obtain \eqref{eq: correlation length 2} by proving that $L(\beta)=\xi(\beta)^{1+o(1)}$. Using \eqref{eq: near-critical upper bound} yields that $\xi(\beta)\leq C_1 L(\beta)$ for some $C_1=C_1(d)>0$. Moreover, by classical sub-additivity arguments\footnote{Indeed, by Griffiths inequality, one has $\langle \tau_0\tau_{(n+k)\mathbf{e}_1}\rangle_\beta\geq \langle \tau_0\tau_{n\mathbf{e}_1}\rangle_\beta\langle \tau_0\tau_{k\mathbf{e}_1}\rangle_\beta$ for every $n,k\geq 0$. The statement follows from Fekete's lemma.}, one has that $\langle \tau_0\tau_{n\mathbf e_1}\rangle_\beta\leq \exp(-n/\xi(\beta))$. Together with the MMS inequalities, this yields
\begin{equation}
	\varphi_\beta(\Lambda_n)\leq C_2n^{d-1}e^{-n/\xi(\beta)},
\end{equation}
which implies that $L(\beta)\leq C_3 \xi(\beta)\log \xi(\beta)$.
\end{proof}

\begin{Rem} When $d\geq 5$, there exists $c_1,C_1>0$ such that for all $\beta<\beta_c$, $c(\beta_c-\beta)^{-1}\leq \chi(\beta)\leq C(\beta_c-\beta)^{-1}$, see \cite{A}. This observation improves \eqref{eq: correlation length 1} and yields the existence of $c_2,C_2>0$ such that for all $\beta<\beta_c$,
	$c_2(\beta_c-\beta)^{-1/2}\leq L(\beta)\leq C_2(\beta_c-\beta)^{-1/2}$.
Combined with \eqref{eq: near-critical upper bound} and Theorem \ref{thm: lower bound}, we obtain the following near-critical estimate on the two-point function: there exist $c,C>0$ such that for all $\beta<\beta_c$, 
\begin{equation}
	\langle \tau_0\tau_x\rangle_\beta\leq \frac{C}{|x|^{d-2}}e^{-c(\beta_c-\beta)^{1/2}|x|}, \qquad x \in \mathbb Z^d\setminus \{0\},
\end{equation}
\begin{equation}
	\langle \tau_0\tau_x\rangle_{\beta}\geq \frac{c}{|x|^{d-2}}, \qquad 1\leq |x|\leq L(\beta).
\end{equation}
Such estimates have been essential for the study of the \emph{torus plateau} in various models of statistical mechanics \cite{slade2023near,hutchcroft2023high,liu2023general,liu2024near,liu2024universal}, and are used in \cite{LiuPanisSlade} to establish it in the case of the Ising model in dimensions $d\geq 5$. 
\end{Rem}

%

\paragraph{Divergence of the bubble diagram.} As described by Aizenman in his seminal work \cite{A}, one may define the \emph{bubble diagram}
\begin{equation}
    B(\beta):= \sum_{x\in \mathbb Z^d}\langle \tau_0\tau_x\rangle_{\beta}^2,
\end{equation}
whose finiteness at criticality allows to establish that some critical exponents exist and take their mean-field values. The condition $B(\beta_c)<\infty$ also implies triviality of the critical scaling limits of the model, see \cite[Theorem~D.2]{Pan23+}. The infrared bound yields that this condition is satisfied for $d\geq 5$. The estimates for $d=2$ imply that the bubble diagram diverges. In the following result, we complete the picture by proving that the bubble diagram diverges for $d=3,4$.
\begin{Thm}[Divergence of the bubble diagram]\label{thm: divergence bubble}
Let $d=3,4$. Then, 
$	B(\beta_c)=\infty.$
\end{Thm}

\begin{proof} Below, the constants $C_i>0$ only depend on $d$. We define
\begin{equation}
	B_n(\beta):=\sum_{x\in \Lambda_n}\langle \tau_0\tau_x\rangle_{\beta}^2.
\end{equation}
Using the Cauchy--Schwarz inequality, one gets
\begin{equation}
	\chi_{4n}(\beta_c)=\sum_{x\in \Lambda_{4n}}\langle \tau_0\tau_x\rangle_{\beta_c}\leq C_1 n^{d/2} \sqrt{B_{4n}(\beta_c)}.
	\end{equation}
Now, by the MMS inequalities,
\begin{equation}\label{eq: cs bubble}
	\sum_{k=1}^{4n}k^{d-1}\langle \tau_0\tau_{k\mathbf{e_1}}\rangle_{\beta_c}^2\leq C_2\sum_{k=1}^{4n}\sum_{|y|=\lfloor k/d\rfloor}\langle \tau_0\tau_y\rangle_{\beta_c}^2\leq C_3dB_{4n}(\beta_c).
\end{equation}
By another application of the Cauchy--Schwarz inequality,
\begin{align}
	n^{d-2}\sum_{k=1}^{4n}k\langle \tau_0\tau_{k\mathbf{e}_1}\rangle_{\beta_c}&=n^{d-2}\sum_{k=1}^{4n}\frac{1}{k^{(d-3)/2}}\cdot k^{(d-1)/2}\langle \tau_0\tau_{k\mathbf{e}_1}\rangle_{\beta_c}\\&\leq C_4 \left\{
    \begin{array}{ll}
        n\cdot\sqrt{n}\cdot\sqrt{B_{4n}(\beta_c)} & \mbox{if } d=3, \\
        n^2\cdot\sqrt{\log n}\cdot\sqrt{B_{4n}(\beta_c)}  & \mbox{if }d=4.
    \end{array}
\right.
\end{align}
Now, using Theorem \ref{thm: pointwise lower bound all dimensions}, for all $n\geq N_1$,
\begin{equation}\label{eq: intrinsic lower bound with the bubble}
	\langle \tau_0\tau_{n\mathbf{e}_1}\rangle_{\beta_c}\geq \frac{C_5}{\beta}\left\{
    \begin{array}{ll}
        \tfrac{1}{n^{3/2}\sqrt{B_{4n}(\beta_c)}} & \mbox{if } d=3, \\
        \tfrac{1}{n^{2}\sqrt{\log n}\sqrt{B_{4n}(\beta_c)}}  & \mbox{if }d=4.
    \end{array}
\right.
\end{equation}
Combining the above displayed equation with \eqref{eq: cs bubble} yields that the bubble cannot be finite.

\end{proof}

\begin{Rem}\label{rem: blow up bubble} Using \eqref{eq: intrinsic lower bound with the bubble}, we may obtain a quantitative estimate on the blow-up of the bubble diagram when $d=3,4$: there exists $c=c(d)>0$ such that, for every $n\geq 2$,
\begin{equation}
	B_{16n}(\beta_c)^2\geq \begin{cases}\displaystyle c \log n &\text{ if }d=3,\\
    \displaystyle c \log \log n  &\text{ if }d=4.\end{cases}
\end{equation}

\end{Rem}

\section{Proof of Theorem \ref{prop: intermediate step lower bound} for the Ising model}

As mentioned above, the proof builds on the random current representation of the Ising model. We briefly recall this classical expansion and refer to \cite{duminil2017lectures,PanisThesis} for a more detailed introduction.
\subsection{The random current representation}\label{Section: rcr} Let $\Lambda$ be a finite subset of $\mathbb Z^d$.
\begin{Def} A \textit{current} $\n$ on $\Lambda$ is a function defined on the set $E(\Lambda):=\lbrace \lbrace x,y\rbrace, \: x,y \in \Lambda \textup{ and }x\sim y \rbrace$ and taking its values in $\mathbb N=\lbrace 0,1,\ldots\rbrace$. We let $\Omega_\Lambda$ be the set of currents on $\Lambda$. The set of \emph{sources} of $\n$, denoted by $\sn$, is defined as
\begin{equation}
\sn:=\Big\{ x \in \Lambda \: : \: \sum_{y\sim x}\n_{x,y}\textup{ is odd}\Big\}.
\end{equation}
We also set 
\begin{equation}
w_\beta(\n):=\prod_{\substack{\lbrace x,y\rbrace\subset \Lambda\\ x\sim y}}\dfrac{\beta^{\n_{x,y}}}{\n_{x,y}!}.
\end{equation}
\end{Def}
If $A\subset \Lambda$, define a probability measure $\mathbf{P}_{\Lambda,\beta}^A$ on $\Omega_\Lambda$ as follows: for a current $\n$ on $\Lambda$,
\begin{equation}
    \mathbf{P}_{\Lambda,\beta}^A[\n]:=\mathds{1}_{\sn=A}\frac{w_\beta(\n)}{Z_{\Lambda,\beta}^A},
\end{equation}
where $Z_{\Lambda,\beta}^A:=\sum_{\sn=A}w_\beta(\n)$ is a normalisation constant. If $\mathcal{E}\subset \Omega_\Lambda$, we will also write $Z_{\Lambda,\beta}^A[\mathcal{E}]:=\sum_{\sn=A}w_\beta(\n)\mathds{1}_{\n\in \mathcal{E}}$.
As proved in \cite{ADCS}, if $A$ is a finite (even) subset of $\mathbb Z^d$, the sequence of probability measures $(\mathbf{P}_{\Lambda,\beta}^A)_{\Lambda\subset \mathbb Z^d}$ admits a weak limit as $\Lambda\nearrow \mathbb Z^d$ that we denote by $\mathbf{P}_{\beta}^A$.

It is possible to expand the correlation functions of the Ising model to relate them to currents: for
$A\subset \Lambda$, if $\sigma_A:=\prod_{x\in A}\sigma_x$ we get

\begin{equation}\label{equation correlation rcr}
\left\langle \sigma_A\right\rangle_{\Lambda,\beta}=\dfrac{\sum_{\sn=A}w_\beta(\n)}{\sum_{\sn=\emptyset}w_\beta(\n)}.
\end{equation}

A current $\n$ with empty source set can be seen as the edge count of a multigraph obtained as a union of loops. Adding sources to a current configuration comes down to adding a collection of paths connecting pairwise the sources.  We will identify $\n$ with its multigraph $\mathcal{N}$ in which the vertex set is $\Lambda$ and where there are exactly $\n_{x,y}$ edges between $x$ and $y$. We will write $w_\beta(\mathcal N):=w_\beta(\n)$ and $\partial \mathcal{N}:=\sn$.

As it turns out, connectivity properties of the multigraph induced by a current will play a crucial role in the analysis of the underlying Ising model. For a multigraph $\mathcal N$ on $\Lambda$ and $x,y\in \Lambda$, we write $x\overset{\mathcal N}{\longleftrightarrow} y$ if $x$ is connected to $y$ in $\mathcal N$.
More generally, if $A,B\subset \Lambda$, we write $A\overset{\mathcal N}{\longleftrightarrow} B$ if there exists $x\in A$ and $y \in B$ such that $x\overset{\mathcal N}{\longleftrightarrow} y$.

The main interest of the above expansion lies in the following result called the \textit{switching principle}, see \cite[Lemma~2.1]{aizenman2019emergent}. This combinatorial result first appeared in \cite{GHS} to prove the concavity of the magnetisation of an Ising model with a positive external field, but the probabilistic picture attached to it flourished in \cite{A}.

\begin{Lem}[Switching principle]\label{switching principle} Let $\Lambda\subset \mathbb Z^d$ and $A\subset \Lambda$. Let $\mathcal K\subset\mathcal{M}$ be multigraphs on $\Lambda$ and $f$ be a bounded function defined over the space of multigraphs on $\Lambda$. Then, one has
\begin{equation}\label{eq: switching principle}
    \sum_{\substack{\mathcal{N}\subset \mathcal{M}\\\partial \mathcal{N}=A}}f(\mathcal{N})=\sum_{\substack{\mathcal{N}\subset \mathcal{M}\\\partial \mathcal{N}=A\Delta \partial\mathcal K}}f(\mathcal{N}\Delta\mathcal{K}).
\end{equation}
\end{Lem}

\begin{proof} The result follows from the observation that the map $\mathcal N\mapsto \mathcal N\Delta \mathcal K$ is a one-to-one map (involution) from $\{ \mathcal N\subset \mathcal M: \partial \mathcal N=A\}$ to $\{ \mathcal N\subset \mathcal M: \partial \mathcal N=A\Delta \partial\mathcal K\}$.
\end{proof}

\subsection{Proof of the theorem}\label{section: proof of the prop for the ising model} We now turn to the proof of Theorem \ref{prop: intermediate step lower bound} in the case of the Ising model. We will follow \cite{aizenman2019emergent} and consider ``folded'' single currents on $\mathbb Z^d$. We partition the edge-set $E(\mathbb Z^d)$ of $\mathbb Z^d$ into three disjoint sets:
\begin{align*}
    E_{-}(\mathbb H_n)&:=\Big\{ \{u,v\}\in E(\mathbb Z^d) \textup{ s.t.~at least one endpoint is strictly on the left of $\mathbb H_n$}\Big\},\\
    E_{+}(\mathbb H_n)&:=\Big\{ \{u,v\}\in E(\mathbb Z^d) \textup{ s.t.~at least one endpoint is strictly on the right of $\mathbb H_n$}\Big\},\\
    E_{0}(\mathbb H_n)&:=E(\mathbb Z^d)\setminus\left(E_{-}(\mathbb H_n)\cup E_{+}(\mathbb H_n)\right).
\end{align*}
Decompose $\n\in \Omega_{\mathbb Z^d}$ into its restrictions $\n_{-}$, $\n_{+}$ and $\n_0$ to the above three subsets of $E(\mathbb Z^d)$. For the purpose of later applying the switching principle, it is convenient to consider the multigraph representation of these objects. Consider the multigraph $\mathcal{M}_n$
 obtained by taking the union of the multigraph $\mathcal{N}_{-}$ associated with $\n_{-}$ and the reflection $\mathcal{R}_n(\mathcal{N}_+)$ of the multigraph $\mathcal{N}_+$ associated with $\n_{+}$. In this context, the switching principle yields the following result that we will use several times below. 
\begin{Lem}[Switching principle for reflected currents]\label{lem: switching for reflected currents} Let $\beta>0$. Let $\Lambda\subset \mathbb Z^d$ be finite and symmetric under $\mathcal{R}_n$. Assume that $A\subset \Lambda$ and $x\in \Lambda$ are strictly on the left of $\mathbb H_n$. Then,
\begin{equation}
    Z_{\Lambda,\beta}^{A}[x\xleftrightarrow[]{\:\mathcal M_n\:} \mathbb H_n]=Z^{A\Delta\lbrace x,\mathcal{R}_n(x)\rbrace}_{\Lambda,\beta}[x\xleftrightarrow[]{\:\mathcal M_n\:} \mathbb H_n].
\end{equation}
\end{Lem}
\begin{proof} By definition,
\begin{equation}\label{eq: proof switching 1}
	Z^{A}_{\Lambda,\beta}[x\xleftrightarrow[]{\:\mathcal M_n\:} \mathbb H_n]=\sum_{\n_0}w_\beta(\n_0)\sum_{\substack{\n_-, \:\n_+\\ \partial(\n_-+\n_+)=\sn_0\Delta A}}\mathds{1}_{x\xleftrightarrow[]{\:\mathcal M_n\:} \mathbb H_n}\,w_\beta(\n_-)w_\beta(\n_+).
\end{equation}
To lighten notations, we will denote $\sum_{\sn_-=A, \: \sn_+=B}$ the sum over pair of currents $(\n_-,\n_+)$ satisfying $(\sn_-,\sn_+)=(A,B)$.
When it exists, call $\Gamma=\Gamma(\mathcal{M}_n,x)$ the smallest path connecting $x$ to $\mathbb H_n$ in $\mathcal M_n$ (according to some fixed ordering). In order to be in the setup of application of Lemma \ref{switching principle}, we decompose \eqref{eq: proof switching 1} according to the sources of $\n_-$ and $\n_+$ as well as $\Gamma$, and adopt the multigraph notations for the currents.  We call $\alpha$ the endpoint of $\Gamma$ and also decompose \eqref{eq: proof switching 1} according to the (unique) value of $\alpha$. Together with \eqref{eq: proof switching 1} we obtain,
\begin{align*}
    Z^{A}_{\Lambda,\beta}[x\xleftrightarrow[]{\:\mathcal M_n\:} \mathbb H_n]   
    &= \sum_{\n_0}w_\beta(\n_0)\sum_{\substack{D\subset \mathbb H_n\cap \Lambda}}\sum_{\substack{\alpha\in \mathbb H_n\\\gamma:x \rightarrow \alpha}}\sum_{\substack{\sn_-=D\Delta \sn_0\Delta A\\\sn_+=D}}w_\beta(\n_-)w_\beta(\n_+)\mathds{1}_{\Gamma=\gamma}
    \\
    &=\sum_{\n_0}w_\beta(\n_0)\sum_{\substack{D\subset \mathbb H_n\cap \Lambda}}\sum_{\substack{\alpha\in \mathbb H_n\\\gamma:x \rightarrow \alpha}}\sum_{\partial \mathcal M_n=\partial\n_0\Delta A}w_\beta(\mathcal M_n)\mathds{1}_{\Gamma=\gamma}\sum_{\substack{\mathcal N\subset \mathcal M_n\\ \partial \mathcal N=D}} 1    \\
    &=\sum_{\n_0}w_\beta(\n_0)\sum_{\substack{D\subset \mathbb H_n\cap \Lambda}}\sum_{\substack{\alpha\in \mathbb H_n\\\gamma:x \rightarrow \alpha}}\sum_{\partial \mathcal M_n=\partial\n_0\Delta A}w_\beta(\mathcal M_n)\mathds{1}_{\Gamma=\gamma}\sum_{\substack{\mathcal N\subset \mathcal M_n\\\partial \mathcal N=D\Delta \{x,\alpha\}}} 1,
\end{align*} 
where we used the switching principle from Lemma \ref{switching principle} in the third line. As a result, we obtained,
\begin{align*}
    Z^{A}_{\Lambda,\beta}[x\xleftrightarrow[]{\:\mathcal M_n\:} \mathbb H_n]&=\sum_{\n_0}w_\beta(\n_0)\sum_{\substack{D\subset \mathbb H_n\cap \Lambda}}\sum_{\substack{\alpha\in \mathbb H_n\\\gamma:x \rightarrow \alpha}}\sum_{\substack{\sn_-=D\Delta \sn_0\Delta A \Delta \{x,\alpha\}\\\sn_+=D\Delta \{\mathcal R_n(x),\alpha\}}}w_\beta(\n_-)w_\beta(\n_+)\mathds{1}_{\Gamma=\gamma}
    \\
    &=\sum_{\n_0}w_\beta(\n_0)\sum_{\substack{\n_-,\n_+\\\partial(\n_-+\n_+)=\sn_0\Delta A\Delta \{x,\mathcal R_n(x)\}}}w_\beta(\n_-)w_\beta(\n_+)\mathds{1}_{x\xleftrightarrow[]{\:\mathcal M_n\:} \mathbb H_n}
   \\
    &=Z^{A\Delta\lbrace x,\mathcal{R}_n(x)\rbrace}_{\Lambda,\beta}[x\xleftrightarrow[]{\:\mathcal M_n\:} \mathbb H_n],
\end{align*}
which concludes the proof.
\end{proof}
\begin{figure}
\begin{center}
\includegraphics[scale=1.0]{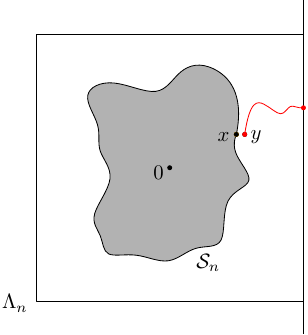}
\put(5,10){$\mathbb H_n$}
\end{center}
\caption{The set $\mathcal{S}_n$ is represented in grey. On its boundary, a point $x$ (in black) is next to a point $y$ (in red) which connects to $\mathbb H_n$ in $\mathcal{M}_n$, i.e.~which does not belong to $\mathcal{S}_n(+\mathbf{e}_1)$.}
\label{figure: Lower bound}
\end{figure}

The above quantities were all defined with respect to the direction $\mathbf e_1$. In order to generalise these objects, let us write $\mathcal{M}_n(+\mathbf{e}_1)=\mathcal{M}_n$, $\mathcal R_n(+\mathbf{e}_1)=\mathcal R_n$, $\mathbb H_n(+\mathbf{e}_1)=\mathbb H_n$, and define
\begin{equation}
    \mathcal{S}_n(+\mathbf{e}_1):=\Big\lbrace x\in \Lambda_n: \: x \overset{\mathcal M_n(+\mathbf e_1)}{\centernot\longleftrightarrow} \mathbb H_n(+\mathbf{e}_1)\Big\rbrace.
\end{equation}
\noindent Similarly, define $\mathbb H_n(\pm \mathbf{e}_i)$, $\mathcal R_n(\pm \mathbf{e}_i)$, $\mathcal{M}_n(\pm \mathbf{e}_i)$, and $\mathcal{S}_n(\pm\mathbf{e}_i)$ for $1\leq i \leq d$ in the other $2d-1$ directions. 
Let (see Figure \ref{figure: Lower bound})
\begin{equation}
    \mathcal{S}_n:=\bigcap_{1\leq i \leq d}\left(\mathcal{S}_n(+\mathbf{e}_i)\cap \mathcal{S}_n(-\mathbf{e}_i)\right).
\end{equation}
Note that by definition, $\mathcal{S}_n\subset \Lambda_{n-1}$. Recall the definitions of $\mathbf P^\emptyset_{\beta}$ and $\varphi_\beta(S)$ from above.

 By definition of $L(\beta)$, if $1\leq n\leq L(\beta)$ and $S\subset \Lambda_{n-1}$, then $\varphi_\beta(S)\geq 1/2$. As a result, if $1\leq n\leq L(\beta)$, then
\begin{equation}\label{eq: proof lower 1}
    \tfrac{1}{2}\mathbf P_\beta^\emptyset[0 \in \mathcal{S}_n]\leq \mathbf E_\beta^\emptyset[\mathds{1}_{0\in \mathcal{S}_n}\varphi_\beta(\mathcal{S}_n)].
\end{equation}
The first step is to observe that the left-hand side in \eqref{eq: proof lower 1} is bounded away from $0$.

\begin{Lem}\label{lem :lemma 1}  Let $d\geq 3$. There exist $c,N_0>0$ such that for $\beta \leq \beta_c$ and $n\geq N_0$,
\begin{equation}
    \mathbf P_\beta^\emptyset[0\in \mathcal{S}_n]\geq c.
\end{equation}
\end{Lem}
\begin{proof}  Using a union bound and the symmetries of $\mathbf P_\beta^\emptyset$,
\begin{equation}
    \mathbf P_\beta^\emptyset[0\notin \mathcal{S}_n]\leq \sum_{i=1}^d(\mathbf P_\beta^\emptyset[0\notin \mathcal{S}_n(+\mathbf{e}_i)]+\mathbf P_\beta^\emptyset[0\notin \mathcal{S}_n(-\mathbf{e}_i)])=2d \mathbf P_\beta^\emptyset[0\notin \mathcal{S}_n(+\mathbf{e}_1)].
\end{equation}
Fix for a moment $\Lambda\subset\mathbb Z^d$ finite, symmetric under $\mathcal{R}_n$, and containing $\Lambda_{4n}$. 
The switching principle for reflected currents (Lemma \ref{lem: switching for reflected currents}) implies
\begin{equation}
    Z^\emptyset_{\Lambda,\beta}[0\notin \mathcal{S}_n(+\mathbf{e}_1)]=Z_{\Lambda,\beta}^{\emptyset}[0\xleftrightarrow[]{\:\mathcal M_n\:}\mathbb H_n]
    =
    \sum_{\n\in \Omega_\Lambda: \:\sn=\emptyset}w_\beta(\n)\mathds{1}_{0\xleftrightarrow[]{\:\mathcal M_n\:} \mathbb H_n}= Z_{\Lambda,\beta}^{\{0,\mathcal{R}_n(0)\}}.
\end{equation}
Dividing the above display by $Z^\emptyset_{\Lambda,\beta}$ we get
\begin{align}
	\mathbf P^{\emptyset}_{\Lambda,\beta}[0\notin \mathcal S_n(+\mathbf{e}_1)]=\mathbf P^\emptyset_{\Lambda,\beta}[0\xleftrightarrow[]{\:\mathcal M_n\:} \mathbb H_n]\leq \frac{Z_{\Lambda,\beta}^{\{0,\mathcal R_n(0)\}}}{Z^\emptyset_{\Lambda,\beta}}&=\langle \sigma_0\sigma_{\mathcal R_n(0)}\rangle_{\Lambda,\beta}
	\leq \langle \sigma_0\sigma_{2n\mathbf{e}_1}\rangle_{\beta},
\end{align}
where in the last inequality we used the monotonicity of the two-point function in $\Lambda$. The proof follows readily from the infrared bound \eqref{eq: Infrared bound} and from taking the limit $\Lambda\nearrow \mathbb Z^d$.
\end{proof}

Combining \eqref{eq: proof lower 1} and Lemma \ref{lem :lemma 1}, we obtain that for some $c,N_0>0$, if $N_0\leq n\leq L(\beta)$,
\begin{equation}\label{eq: proof lower 2}
    \frac{c}{2\beta}\leq \mathbf E^\emptyset_\beta\Big[\mathds{1}_{0\in \mathcal{S}_n}\sum_{\substack{x\in \mathcal{S}_n\\y\sim x, \:y\notin \mathcal{S}_n}}\langle \sigma_0\sigma_x\rangle_{\mathcal{S}_n,\beta}\Big].
\end{equation}
Notice that if $x\in \mathcal{S}_n$, $y\sim x$ and $y\notin \mathcal{S}_n$, then there exist $1\leq i \leq d$ and $\varepsilon\in \lbrace \pm 1\rbrace$ such that $y\notin \mathcal{S}_n(\varepsilon \mathbf{e}_i)$. As a result,
\begin{align*}
    \mathds{1}_{0\in \mathcal{S}_n}\sum_{\substack{x\in \mathcal{S}_n\\y\sim x, \:y\notin \mathcal{S}_n}}\langle
    \sigma_0\sigma_x\rangle_{\mathcal{S}_n,\beta}
    &\leq 
    \mathds{1}_{0\in 
    \mathcal{S}_n}\sum_{i=1}^d\sum_{\varepsilon\in \lbrace \pm 1\rbrace}\sum_{\substack{x\in \mathcal{S}_n\\y\sim x, \:y\notin \mathcal{S}_n(\varepsilon\mathbf{e}_i)}}\langle 
    \sigma_0\sigma_x\rangle_{\mathcal{S}_n,\beta}
\\&\leq \mathds{1}_{0\in 
    \mathcal{S}_n}
    \sum_{i=1}^d\sum_{\varepsilon\in \lbrace \pm 1\rbrace}\sum_{\substack{x\in \mathcal{S}_n(\varepsilon\mathbf{e}_i)\cap \Lambda_{n-1} \\y\sim x, \:y\notin \mathcal{S}_n(\varepsilon\mathbf{e}_i)}}\langle 
    \sigma_0\sigma_x\rangle_{\mathcal{S}_n(\varepsilon\mathbf{e}_i)\cap \Lambda_{n-1},\beta},
\end{align*}
where we used that $\mathcal{S}_n\subset \mathcal{S}_n(\varepsilon\mathbf{e}_i)\cap \Lambda_{n-1}$ together with Griffiths' inequality on the second line. Using the symmetries of $\mathbf P_\beta^\emptyset$ and \eqref{eq: proof lower 2}, we get
\begin{equation}\label{eq: proof lower bound 3}
    \frac{c}{4d\beta}\leq \mathbf E_\beta^\emptyset\Big[\mathds{1}_{0\in \mathcal{S}_n^1}\sum_{\substack{x\in \mathcal{S}_n^1\\y\sim x, \:y\in \Lambda_n}}\mathds{1}_{y\xleftrightarrow[]{\:\mathcal M_n\:} \mathbb H_n}\langle \sigma_0\sigma_x\rangle_{\mathcal{S}_n^1,\beta}\Big],
\end{equation}
where $\mathcal{S}_n^1:=\mathcal{S}_n(+\mathbf{e}_1)\cap \Lambda_{n-1}$. Theorem \ref{prop: intermediate step lower bound} follows from \eqref{eq: proof lower bound 3} and Lemma \ref{lem: lower proof 3} below.
\begin{Lem}\label{lem: lower proof 3} For $\beta\leq \beta_c$,
\begin{equation}
    \mathbf E^\emptyset_\beta\Big[\mathds{1}_{0\in \mathcal{S}_n^1}\sum_{\substack{x\in \mathcal{S}_n^1\\y\sim x, \:y\in \Lambda_n}}\mathds{1}_{y\xleftrightarrow[]{\:\mathcal M_n\:}\mathbb H_n}\langle \sigma_0\sigma_x\rangle_{\mathcal{S}_n^1,\beta}\Big]\leq \sum_{\substack{x,y\in \Lambda_n\\y\sim x}}(\langle \sigma_0\sigma_x\rangle_\beta-\langle \sigma_0\sigma_{\mathcal{R}_n(x)}\rangle_\beta)\langle \sigma_y \sigma_{\mathcal{R}_n(y)}\rangle_\beta.
\end{equation}
\end{Lem}
\begin{proof} Fix $\Lambda\subset \mathbb Z^d$ finite, symmetric under $\mathcal{R}_n$, which contains $\Lambda_{4n}$, and set $\Lambda^-:=\{x\in \Lambda:x_1< n\}$. Observe that 
\begin{equation}
    \mathbf E^\emptyset_{\Lambda,\beta}\Big[\mathds{1}_{0\in \mathcal{S}_n^1}\sum_{\substack{x\in \mathcal{S}_n^1\\y\sim x, \:y\in \Lambda_n}}\mathds{1}_{y\xleftrightarrow[]{\:\mathcal M_n\:} \mathbb H_n}\langle \sigma_0\sigma_x\rangle_{\mathcal{S}_n^1,\beta}\Big]
    =
    \sum_{\substack{x,y\in \Lambda_n\\y\sim x}}\mathbf E^{\emptyset}_{\Lambda,\beta}\Big[\mathds{1}_{ 0,x\in \mathcal{S}_n^1,\: y \xleftrightarrow[]{\:\mathcal M_n\:} \mathbb H_n}\langle\sigma_0\sigma_x\rangle_{\mathcal{S}_n^1,\beta}\Big].
\end{equation}
Set $\mathcal{G}_n$ to be the union of the set $\{x\in \Lambda^-:x\overset{\mathcal M_n}{\centernot\longleftrightarrow} \mathbb H_n\}$ and its image under the reflection $\mathcal R_n$. Note that $\mathcal G_n\cap \Lambda_{n-1}=\mathcal S_n^1$. For $x \in \Lambda_{n}$ and $y\sim x$ with $y\in \Lambda_n$, write
\begin{align}
    \mathbf E^{\emptyset}_{\Lambda,\beta}\Big[\mathds{1}_{ 0,x\in \mathcal{S}_n^1,\: y \xleftrightarrow[]{\:\mathcal M_n\:} \mathbb H_n}\langle\sigma_0\sigma_x\rangle_{\mathcal{S}_n^1,\beta}\Big]
    &\leq \mathbf E^{\emptyset}_{\Lambda,\beta}\Big[\mathds{1}_{ 0,x\in \mathcal{G}_n,\: y \xleftrightarrow[]{\:\mathcal M_n\:} \mathbb H_n}\langle\sigma_0\sigma_x\rangle_{\mathcal{G}_n,\beta}\Big]
    \notag\\ \label{eq: final lemma 1} &=
   \sum_{\substack{S\subset \Lambda\\S\ni 0,x\\ S\not\ni y}}\mathbf E^{\emptyset}_{\Lambda,\beta}\left[\mathds 1_{\mathcal{G}_n=S
   } \frac{Z^{\{0,x\}}_{S,\beta}}{Z^\emptyset_{S,\beta}}\right].
\end{align} 
Now, if $\n$ is such that $\mathcal G_n=S$, and $E(S,\Lambda\setminus S)$ denotes the set of edges between $S$ and its complement, one must have that $\n_e=0$ for all $e\in E(S,\Lambda\setminus S)$. With a small abuse of notation, we let $Z^\emptyset_{E(\Lambda\setminus S)\cup E(S,\Lambda\setminus S),\beta}$ denote the partition function of sourceless currents defined on the union of edge sets $E(\Lambda\setminus S)\cup E(S,\Lambda\setminus S)$. Then,
\begin{align}
    Z^\emptyset_{\Lambda,\beta}[\mathcal{G}_n=S]&=Z^\emptyset_{S,\beta}Z^\emptyset_{E(\Lambda\setminus S)\cup E(S,\Lambda\setminus S),\beta}[\mathcal{G}_n=S, \: \n_e=0 \: \:\forall e\in E(S,\Lambda\setminus S)]\label{eq:g_n argument 1}
    \\ Z^{\{0,x\}}_{\Lambda,\beta}[\mathcal{G}_n=S]&=Z^{\{0,x\}}_{S,\beta}Z^\emptyset_{E(\Lambda\setminus S)\cup E(S,\Lambda\setminus S),\beta}[\mathcal{G}_n=S, \: \n_e=0 \: \:\forall e\in E(S,\Lambda\setminus S)].\label{eq:g_n argument 2}
\end{align}

As a result,
\begin{equation}
    \mathbf E^{\emptyset}_{\Lambda,\beta}\Big[\mathds 1_{\mathcal{G}_n=S} \frac{Z^{\{0,x\}}_{S,\beta}}{Z^\emptyset_{S,\beta}}\Big]=\frac{Z^\emptyset_{\Lambda,\beta}[\mathcal G_n=S]}{Z^\emptyset_{\Lambda,\beta}} \frac{Z^{\{0,x\}}_{S,\beta}}{Z^\emptyset_{S,\beta}}=\frac{Z^{\{0,x\}}_{\Lambda,\beta}[\mathcal{G}_n=S]}{Z^\emptyset_{\Lambda,\beta}}.
\end{equation}
We rewrite the right-hand side of \eqref{eq: final lemma 1} as
\begin{equation}\label{eq:ha}
    \sum_{\substack{S\subset \Lambda\\S\ni 0,x\\ S\not\ni y}}\frac{Z^{\{0,x\}}_{\Lambda,\beta}[\mathcal{G}_n=S]}{Z^\emptyset_{\Lambda,\beta}}=\frac{Z^{\{0,x\}}_{\Lambda,\beta}[0,x \overset{\mathcal M_n}{\centernot\longleftrightarrow}\mathbb H_n, \: y\xleftrightarrow[]{\:\mathcal M_n\:} \mathbb H_n 
    ]}{Z^\emptyset_{\Lambda,\beta}}.
\end{equation}
Now, we analyse $Z^{\{0,x\}}_{\Lambda,\beta}[0,x\overset{\mathcal M_n}{\centernot\longleftrightarrow} \mathbb H_n, \: y\xleftrightarrow[]{\:\mathcal M_n\:}\mathbb H_n]$ by conditioning on the set $\mathcal{C}_n(0)$  defined as the union of the cluster of $0$ in $\mathcal{M}_n$ and its reflection with respect to $\mathbb H_n$. Write,
\begin{equation}\label{eq:proof1}
    Z^{\{0,x\}}_{\Lambda,\beta}[0,x\overset{\mathcal M_n}{\centernot\longleftrightarrow} \mathbb H_n, \: y\xleftrightarrow[]{\:\mathcal M_n\:}\mathbb H_n]=\sum_{\substack{S\cap \mathbb H_n=\emptyset\\ S\ni 0,x\\S\not\ni y}}Z^{\{0,x\}}_{\Lambda,\beta}[\mathcal{C}_n(0)=S, \: y\xleftrightarrow[]{\:\mathcal M_n\:}\mathbb{H}_n ].
\end{equation}
Finally,
\begin{equation}\label{eq:proof2}
    Z^{\{0,x\}}_{\Lambda,\beta}[\mathcal{C}_n(0)=S, \: y\xleftrightarrow[]{\:\mathcal M_n\:} \mathbb{H}_n]= Z^{\{0,x\}}_{\Lambda,\beta}[\mathcal{C}_n(0)=S]\frac{Z_{\Lambda\setminus S,\beta}^\emptyset[y\xleftrightarrow[]{\:\mathcal M_n\:} \mathbb H_n]}{Z^\emptyset_{\Lambda\setminus S,\beta}}.
\end{equation}
Applying the switching principle for reflected currents (Lemma \ref{lem: switching for reflected currents}),
\begin{equation}\label{eq:proof3}
    Z_{\Lambda\setminus S,\beta}^\emptyset[y\xleftrightarrow[]{\:\mathcal M_n\:} \mathbb H_n]= Z^{\{y,\mathcal{R}_n(y)\}}_{\Lambda\setminus S,\beta}.
\end{equation}
Moreover, using one last time Lemma \ref{lem: switching for reflected currents},
\begin{equation}\label{eq:proof4}
    \sum_{\substack{S\cap \mathbb H_n=\emptyset\\
     \: S\ni 0,x}}Z^{\{0,x\}}_{\Lambda,\beta}[\mathcal{C}_n(0)=S]=Z^{\{0,x\}}_{\Lambda,\beta}[0,x\overset{\mathcal M_n}{\centernot\longleftrightarrow}\mathbb H_n]
     = Z^{\{0,x\}}_{\Lambda,\beta}-Z^{\{0,\mathcal{R}_n(x)\}}_{\Lambda,\beta}.
\end{equation}
Putting \eqref{eq:proof1}--\eqref{eq:proof4} in \eqref{eq:ha} and using Griffiths' inequality to replace $\Lambda\setminus S$ by $\Lambda$, we get
\begin{equation}
    \mathbf E^{\emptyset}_{\Lambda,\beta}\Big[\mathds{1}_{ 0,x\in \mathcal{S}_n^1,\: y \xleftrightarrow[]{\:\mathcal M_n\:} \mathbb H_n }\langle\sigma_0\sigma_x\rangle_{\mathcal{S}_n^1,\beta}\Big]
    \leq
   \Big(\langle \sigma_0\sigma_x\rangle_{\Lambda,\beta}-\langle \sigma_0\sigma_{\mathcal{R}_n(x)}\rangle_{\Lambda,\beta}\Big)\langle \sigma_y \sigma_{\mathcal{R}_n(y)}\rangle_{\Lambda,\beta}.
\end{equation}
We conclude by taking $\Lambda \nearrow \mathbb Z^d$.
\end{proof}

\begin{Rem}\label{rem: 1/2 to 1/4} Note that if one replaces $L(\beta)$ with $L'(\beta)$ defined by
\begin{equation}\label{eq:def L'}
	L'(\beta):=\inf\left\lbrace k\geq 1: \:\exists S\subset \mathbb Z^d,\: 0\in S, \: \textup{diam}(S)\leq 2k, \: \varphi_{\rho,\beta}(S)< 1/4\right\rbrace,
\end{equation}
then \eqref{eq: proof lower 1} holds with $\tfrac{1}{2}$ replaced by $\tfrac{1}{4}$. This means that, to the cost of changing $c_0$, we may repeat the above argument to obtain a version of Theorem \ref{prop: intermediate step lower bound} which holds with $L(\beta)$ replaced by $L'(\beta)$. This will be used in Section \ref{section: extension to all models in the gs class}.
\end{Rem}

\section{Proof of Theorem \ref{prop: intermediate step lower bound} for the \texorpdfstring{$\varphi^4$}{} model}
As explained in the introduction, the discrete $\varphi^4$ model can be recovered from taking the limit of well-chosen Ising models. In fact, with this procedure we may recover more models which we now describe.

\subsection{The Griffiths--Simon class of measures}\label{section: Gs class of measures}

\begin{Def}[The Griffiths--Simon class of measures]\label{def: gs class} A Borel measure $\rho$ on $\mathbb R$ falls into the \emph{Griffiths--Simon (GS) class} of measures if it satisfies one of the following conditions:
\begin{enumerate}
    \item[$(i)$] there exists an integer $N\geq 1$, a renormalisation constant $Z>0$, $(J_{i,j})_{1\leq i,j\leq N}\in (\mathbb R^+)^{N^2}$, and $(Q_n)_{1\leq n\leq N}\in (\mathbb R^+){N}$ such that for every $F:\mathbb R\rightarrow\mathbb R^+$ bounded and measurable,
    \begin{equation}
        \int_{\mathbb R} F(\tau)\textup{d}\rho(\tau)=\frac{1}{Z}\sum_{\sigma\in \lbrace \pm 1\rbrace^N}F\left(\sum_{n=1}^N Q_n\sigma_n\right)\exp\left(\sum_{i,j=1}^N J_{i,j}\sigma_i\sigma_j\right),
    \end{equation}
    \item[$(ii)$] the measure $\rho$ can be obtained as a weak limit of probability measures of the above type, and it is of sub-Gaussian growth: for some $\alpha>2$,
    \begin{equation}
        \int_{\mathbb R}e^{|\tau|^\alpha}\textup{d}\rho(\tau)<\infty.
    \end{equation}
\end{enumerate}
Measures that fulfill condition $(i)$ are described as being of the ``Ising type'', and those who satisfy $(ii)$ are called ``general''.
\end{Def}
The following result was proved in \cite{simon1973varphi} (see also \cite{krachun2023scaling}).
\begin{Prop}\label{prop: phi4 is GS} Let $g>0$ and $a\in \mathbb R$. The probability measure $\rho_{g,a}$ on $\mathbb R$ given by 
\begin{equation}
    \textup{d}\rho_{g,a}(\varphi)=\frac{1}{z_{g,a}}e^{-g\varphi^4-a\varphi^2}\textup{d}\varphi,
\end{equation}
where $z_{g,a}$ is a renormalisation constant, belongs to the GS class.
\end{Prop}

\subsection{Random current representation of Ising-type models in the GS class}

Fix $\rho$ in the GS class of the Ising-type, and $\beta>0$. The measure $\langle \cdot \rangle_{\Lambda,\rho,\beta}$ can be represented as an Ising measure\footnote{In this measure, the spins $\sigma_{(x,i)}$ and $\sigma_{(x,j)}$ interact through the coupling constant $J_{i,j}$, and for when $x\sim y$, the spins $\sigma_{(x,i)}$ and $\sigma_{(y,i)}$ interact through the coupling constant $\beta Q_iQ_j$.} on $\Lambda^{(N)}:=\Lambda\times K_N$ that we denote by $\langle \cdot \rangle_{\Lambda^{(N)},\rho,\beta}$. In that case, we identify $\tau_x$ with averages of the form
\begin{equation} 
    \sum_{i=1}^N Q_i\sigma_{(x,i)},
\end{equation}
where $Q_i\geq 0$ for $1\leq i \leq N$. For $x\in \mathbb Z^d$, we will denote $B_x:=\lbrace (x,i):\: 1\leq i \leq N\rbrace$.
This point of view allows us to use the random current representation. We introduce a measure $\mathsf{P}^{xy}_{\Lambda,\rho,\beta}$ on $\Omega_{\Lambda\times K_N}$ which we define in the following two steps procedure:
\begin{enumerate}
    \item[-] first, sample two integers $1\leq i,j\leq N$ with probability given by
    \begin{equation}
        \frac{Q_iQ_j\langle \sigma_{(x,i)}\sigma_{(y,j)}\rangle_{\Lambda^{(N)},\rho,\beta}}{\langle \tau_x\tau_y\rangle_{\Lambda,\rho,\beta}},
    \end{equation}
    \item[-] then, sample a current according to the ``standard'' current measure $\mathbf{P}^{\lbrace (x,i),(y,j)\rbrace}_{\Lambda^{(N)},\rho,\beta}$ introduced above.
\end{enumerate}
As before, it is possible to define the infinite volume version of the above measure. We will denote it $\mathsf P^{xy}_{\rho,\beta}$.
Samples of $\mathsf P^{xy}_{\rho,\beta}$ are random currents with \emph{random sources} in $B_x$ and $B_y$. The interest of this measure lies in the fact that it allows to derive bounds on connection probabilities that are formulated in terms of the correlation functions of the field variable $\tau$.
\subsection{Proof of the theorem for Ising-type models in the GS class}
We will prove the following result which is a small modification of Theorem \ref{prop: intermediate step lower bound}. In particular, notice the introduction of the variable $y'$. 

\begin{Prop}\label{prop: key inequality GS class} Let $d\geq 3$. There exist $c_2,N_2>0$ such that the following holds. For every $\rho$ of the Ising type in the GS class, for every $\beta\leq \beta_c(\rho)$, and every $N_2\leq n \leq L(\beta,\rho)$,
\begin{equation}
    c_2\leq \beta\sum_{\substack{x\in \Lambda_n\\y\sim x\\ y'\sim y}}\left(\langle \tau_0\tau_x\rangle_{\rho,\beta}-\langle\tau_0\tau_{\mathcal{R}_n(x)}\rangle_{\rho,\beta}\right)\langle \tau_{y}\tau_{\mathcal{R}_n(y')}\rangle_{\rho,\beta}.
\end{equation}
\end{Prop}
It is easy to deduce Theorem \ref{prop: intermediate step lower bound} from Proposition \ref{prop: key inequality GS class} for measures of the Ising type in the GS class.

\begin{proof}[Proof of Theorem \textup{\ref{prop: intermediate step lower bound}} for measures of the Ising type in the GS class] Let $d\geq 3$ and let $\rho$ be a measure of the Ising type in the GS class. First, by the MMS inequalities, one has that 
\begin{equation}
\langle \tau_y\tau_{\mathcal R_n(y')}\rangle_{\rho,\beta}\leq \max(\langle \tau_y\tau_{\mathcal R_n(y)}\rangle_{\rho,\beta},\langle \tau_{y'}\tau_{\mathcal R_n(y')}\rangle_{\rho,\beta})\leq \langle \tau_y\tau_{\mathcal R_n(y)}\rangle_{\rho,\beta}+ \langle \tau_{y'}\tau_{\mathcal R_n(y')}\rangle_{\rho,\beta}
\end{equation}
 for $y\sim y'$. As a consequence, we get that
\begin{multline}
	\sum_{\substack{x\in \Lambda_n\\y\sim x\\ y'\sim y}}\left(\langle \tau_0\tau_x\rangle_{\rho,\beta}-\langle\tau_0\tau_{\mathcal{R}_n(x)}\rangle_{\rho,\beta}\right)\langle \tau_{y}\tau_{\mathcal{R}_n(y')}\rangle_{\rho,\beta}\\\leq 4d\sum_{\substack{x\in \Lambda_{n+1}\\y\sim x}}\left(\langle \tau_0\tau_x\rangle_{\rho,\beta}-\langle\tau_0\tau_{\mathcal{R}_n(x)}\rangle_{\rho,\beta}\right)\langle \tau_{y}\tau_{\mathcal{R}_n(y)}\rangle_{\rho,\beta}.
\end{multline}
The result follows readily from Proposition \ref{prop: key inequality GS class}.
\end{proof}

\begin{Rem}\label{rem: uniform bound in the GS class}
Note that Theorem \ref{prop: intermediate step lower bound} holds for values of $c_1,N_1$ which are uniform for every $\rho$ of the Ising type in the GS class. This observation will be useful later.
\end{Rem}

We begin with some notations: for a subset $\mathcal{E}$ of $\Omega_{\Lambda\times K_N}$, set
\begin{equation}
    \mathsf{Z}^\emptyset_{\Lambda,\rho,\beta}[\mathcal{E}]:=Z^\emptyset_{\Lambda_N,\rho,\beta}[\mathcal{E}]
  \text{ and }  \mathsf{Z}^{\{u,v\}}_{\Lambda,\rho,\beta}[\mathcal{E}]=\sum_{i,j=1}^NQ_iQ_jZ^{\{(u,i),(v,j)\}}_{\Lambda^{(N)},\rho,\beta}[\mathcal{E}].
\end{equation}
Note that
\begin{equation}
    \frac{\mathsf{Z}^{\{u,v\}}_{\Lambda,\rho,\beta}}{\mathsf{Z}^{\emptyset}_{\Lambda,\rho,\beta}}=\langle \tau_u\tau_v\rangle_{\Lambda,\rho,\beta},
\end{equation}
and
\begin{equation}
    \mathsf P_{\Lambda,\rho,\beta}^{\{u,v\}}[\mathcal{E}]=\sum_{i,j=1}^N\frac{Q_iQ_j\langle \sigma_{(u,i)}\sigma_{(v,j)}\rangle_{\Lambda^{(N)},\rho,\beta}}{\langle \tau_u\tau_v\rangle_{\Lambda,\rho,\beta}}\mathbf{P}^{\{(u,i),(v,j)\}}_{\Lambda^{(N)},\rho,\beta}[\mathcal{E}]=\frac{\mathsf{Z}^{\{u,v\}}_{\Lambda,\rho,\beta}[\mathcal{E}]}{\mathsf{Z}_{\Lambda,\rho,\beta}^{\{u,v\}}}.
\end{equation}
Moreover, for $G_N\subset \Lambda^{(N)}$ containing $B_u$ and $B_v$, we let
\begin{equation}
\mathsf{Z}^{\{u,v\}}_{G_N,\rho,\beta}:=\sum_{i,j=1}^N Q_iQ_j Z^{\{(u,i),(v,j)\}}_{G_N,\rho,\beta}.
\end{equation}
Recall that $\mathbb H_n^{(N)}=\mathbb H_n\times K_N$ and that for $\Lambda \subset \mathbb Z^d$, $\Lambda^-=\{x\in \Lambda:x_1< n\}$. We start by the switching principle in this context.
\begin{Lem}[Switching principle for reflected currents of Ising-type models in the GS class]\label{lem: switching principle reflected currents gs class} Let $\Lambda\subset \mathbb Z^d$ finite and symmetric with respect to $\mathcal{R}_n$. Let $x\in \Lambda^-$. Then,
\begin{itemize}
    \item[$(i)$] $\mathsf{Z}^{\emptyset}_{\Lambda,\rho,\beta}[B_x\xleftrightarrow[]{\:\mathcal M_n\:} \mathbb H_n^{(N)}]\leq \beta \sum_{x'\sim x}\mathsf{Z}^{\{x,\mathcal{R}_n(x')\}}_{\Lambda,\rho,\beta}$,
    \item[$(ii)$] $\mathsf{Z}^{\{0,x\}}_{\Lambda,\rho,\beta}[\sn\cap B_x\xleftrightarrow[]{\:\mathcal M_n\:} \mathbb H_n^{(N)}]=\mathsf{Z}^{\{0,\mathcal{R}_n(x)\}}_{\Lambda,\rho,\beta}$
\end{itemize}
\end{Lem}
\begin{proof} For $(i)$, we follow the strategy employed in \cite[Lemma~A.7]{ADC}. Write,
\begin{align*}
    \mathsf Z^{\emptyset}_{\Lambda,\rho,\beta}[B_x\xleftrightarrow[]{\:\mathcal M_n\:} \mathbb H_n^{(N)}]&\leq\sum_{i=1}^N Z^{\emptyset}_{\Lambda^{(N)},\rho,\beta}[(x,i)\xleftrightarrow[]{\:\mathcal M_n\:} \mathbb H_n^{(N)}]
    \\
    &\leq
    \sum_{x'\sim x}\sum_{i,j=1}^N (\beta Q_iQ_j) Z^{\{(x,i),(x',j)\}}_{\Lambda^{(N)},\rho,\beta}[(x',j)\xleftrightarrow[]{\:\mathcal M_n\:} \mathbb H_n^{(N)}]
    \\
    &\leq
    \sum_{x'\sim x}\sum_{i,j=1}^N  (\beta Q_iQ_j) Z^{\{(x,i),(\mathcal{R}_n(x'),j)\}}_{\Lambda^{(N)},\rho,\beta},
\end{align*}
where on the second line we used the fact that any path going out of $B_x$ has to visit $B_{x'}$ for some $x'\sim x$, and on the third line we used\footnote{For full disclosure, we actually used a straightforward generalisation of Lemma \ref{lem: switching for reflected currents} to graphs of the form $\Lambda\times K_N$ that are symmetric with respect to $\mathcal{R}_n$, in the sense that for all $x\in \Lambda$ and all $1\leq i \leq N$, $(x,i)\in \Lambda^{(N)}$ if and only if $\mathcal{R}_n((x,i)):=(\mathcal{R}_n(x),i)\in \Lambda^{(N)}$.} Lemma \ref{lem: switching for reflected currents}. 

For $(ii)$, we essentially reproduce the argument of Lemma~\ref{lem: switching for reflected currents} and omit the proof. 
\end{proof}
With this result, we obtain Proposition \ref{prop: key inequality GS class} by copying the argument used above. We will modify the notations above to make them adapted to our setup. We now define,
\begin{equation}
    \mathcal{S}_n(+\mathbf{e}_1):=\Big\lbrace x\in \Lambda_n: \: B_x\overset{\mathcal M_n(+\mathbf e_1)}{\centernot\longleftrightarrow} \mathbb H_n(+\mathbf{e}_1)\times K_N\Big\rbrace,
\end{equation}
and modify accordingly the definitions of $\mathcal{S}_n$ and $\mathcal{S}_n^1$.
\begin{proof}[Proof of Proposition \textup{\ref{prop: key inequality GS class}}]
With the help of the previous lemma, the proof is basically the same as for the Ising case. The only modification lies in the statement and in the proof of the adaptation of Lemma~\ref{lem: lower proof 3}, which reads as follows.
\end{proof}
\begin{Lem}\label{lem: lower proof GS final} For $\beta\leq \beta_c(\rho)$,
\begin{multline}
    \mathsf E^\emptyset_{\rho,\beta}\Big[\mathds{1}_{0\in \mathcal{S}_n^1}\sum_{\substack{x\in \mathcal{S}_n^1\\y\sim x, \:y\in \Lambda_n}}\mathds{1}_{B_y\xleftrightarrow[]{\:\mathcal M_n\:} \mathbb H_n^{(N)}}\langle \tau_0\tau_x\rangle_{\mathcal{S}_n^1,\rho,\beta}\Big]\\\leq \beta \sum_{\substack{x,y\in \Lambda_n\\y\sim x\\ y'\sim y}}\left(\langle \tau_0\tau_x\rangle_{\rho,\beta}-\langle \tau_0\tau_{\mathcal{R}_n(x)}\rangle_{\rho,\beta}\right)\langle \tau_y \tau_{\mathcal{R}_n(y')}\rangle_{\rho,\beta}.
\end{multline}
\end{Lem}

\begin{proof} Fix $\Lambda\subset \mathbb Z^d$ finite, symmetric under $\mathcal{R}_n$, which contains $\Lambda_{4n}$. Observe that 
\begin{equation}
    \mathsf E^\emptyset_{\Lambda,\rho,\beta}\Big[\mathds{1}_{0\in \mathcal{S}_n^1}\sum_{\substack{x\in \mathcal{S}_n^1\\y\sim x, \:y\in \Lambda_n}}\mathds{1}_{B_y\xleftrightarrow[]{\:\mathcal M_n\:} \mathbb H_n^{(N)}}\langle \tau_0\tau_x\rangle_{\mathcal{S}_n^1,\rho,\beta}\Big]
    =
    \sum_{\substack{x,y\in \Lambda_n\\y\sim x}}\mathsf E^{\emptyset}_{\Lambda,\beta}\Big[\mathds{1}_{ 0,x\in \mathcal{S}_n^1,\: B_y \xleftrightarrow[]{\:\mathcal M_n\:} \mathbb H_n^{(N)}}\langle\tau_0\tau_x\rangle_{\mathcal{S}_n^1,\beta}\Big].
\end{equation}
For $x \in \Lambda_{n}$ and $y\sim x$ with $y\in \Lambda_n$, write 
\begin{align}
    \mathsf E^{\emptyset}_{\Lambda,\beta}\Big[\mathds{1}_{ 0,x\in \mathcal{S}_n^1,\: B_y \xleftrightarrow[]{\:\mathcal M_n\:} \mathbb H_n^{(N)}}\langle\tau_0\tau_x\rangle_{\mathcal{S}_n^1,\rho,\beta}\Big]
    &\leq \mathsf E^{\emptyset}_{\Lambda,\beta}\Big[\mathds{1}_{ B_0,B_x\subset \mathcal{G}_n,\: B_y \xleftrightarrow[]{\:\mathcal M_n\:} \mathbb H_n^{(N)}}\langle\tau_0\tau_x\rangle_{\mathcal{G}_n,\beta}\Big]
    \notag\\ \label{eq: final lemma 1 GS} &=
   \sum_{\substack{S\subset \Lambda^{(N)}\\S\supset B_0,B_x\\ S\not\supset B_y}}\mathsf E^{\emptyset}_{\Lambda,\rho,\beta}\left[\mathds 1_{\mathcal{G}_n=S
   } \frac{\mathsf Z^{\{0,x\}}_{S,\rho,\beta}}{\mathsf Z^\emptyset_{S,\rho,\beta}}\right],
\end{align} 
where
$\mathcal{G}_n$ is the union of $ \big\lbrace (x,i)\in \Lambda^{(N)}: \: (x,i) \overset{\mathcal M_n}{\centernot\longleftrightarrow} \mathbb{H}_n^{(N)}\big\rbrace$ and its image with respect to $\mathcal R_n$.
Reasoning as in \eqref{eq:g_n argument 1} and \eqref{eq:g_n argument 2}, we obtain
\begin{equation}
    \mathsf E^{\emptyset}_{\Lambda,\rho,\beta}\left[\mathds 1_{\mathcal{G}_n=S
   } \frac{\mathsf Z^{\{0,x\}}_{S,\rho,\beta}}{\mathsf Z^\emptyset_{S,\rho,\beta}}\right]
   =
   \frac{\mathsf Z^{\{0,x\}}_{\Lambda,\rho,\beta}[\mathcal{G}_n=S]}{\mathsf Z^\emptyset_{\Lambda,\rho,\beta}}.
\end{equation}
The right-hand side of \eqref{eq: final lemma 1 GS} rewrites
\begin{equation}
    \sum_{\substack{S\subset \Lambda\times K_N\\S\supset B_0,B_x\\ S\not\supset B_y}}\frac{\mathsf Z^{\{0,x\}}_{\Lambda,\rho,\beta}[\mathcal{G}_n=S]}{\mathsf Z^\emptyset_{\Lambda,\rho,\beta}}=\frac{\mathsf Z^{\{0,x\}}_{\Lambda,\rho,\beta}[B_0,B_x\overset{\mathcal M_n}{\centernot\longleftrightarrow} \mathbb H_n^{(N)}, \: B_y\xleftrightarrow[]{\:\mathcal M_n\:} \mathbb H_n^{(N)}]}{\mathsf Z^\emptyset_{\Lambda,\rho,\beta}}.
\end{equation}
Now, we analyse $\mathsf Z^{\{0,x\}}_{\Lambda,\rho,\beta}[B_0,B_x\overset{\mathcal M_n}{\centernot\longleftrightarrow} \mathbb H_n^{(N)}, \: B_y\xleftrightarrow[]{\:\mathcal M_n\:} \mathbb H_n^{(N)}]$ by conditioning on the cluster of $B_0\cup B_x$ in $\mathcal{M}_n$ that we denote by $\mathcal{C}_n(0,x)$. Write,
\begin{multline}
    \mathsf Z^{\{0,x\}}_{\Lambda,\rho,\beta}[B_0,B_x\mathrel{\mathop{\centernot\longleftrightarrow}^{\mathcal M_n}} \mathbb H_n^{(N)}, \: B_y\xleftrightarrow[]{\:\mathcal M_n\:} \mathbb H_n^{(N)}]\\=\sum_{\substack{S\cap \mathbb H_n^{(N)}=\emptyset\\S\supset B_0,B_x}}\mathsf Z^{\{0,x\}}_{\Lambda,\rho,\beta}[\mathcal{C}_n(0,x)=S, \: B_y\xleftrightarrow[]{\:\mathcal M_n\:} \mathbb{H}_n^{(N)}].
\end{multline}
Finally,
\begin{equation}
    \mathsf Z^{\{0,x\}}_{\Lambda,\rho,\beta}[\mathcal{C}_n(0,x)=S, \: B_y\xleftrightarrow[]{\:\mathcal M_n\:} \mathbb{H}_n^{(N)}]= \mathsf Z^{\{0,x\}}_{\Lambda,\rho,\beta}[\mathcal{C}_n(0,x)=S]\frac{\mathsf Z_{\Lambda\setminus S,\rho, \beta}^\emptyset[B_y\xleftrightarrow[]{\:\mathcal M_n\:} \mathbb H_n^{(N)}]}{\mathsf Z^\emptyset_{\Lambda\setminus S,\rho,\beta}}.
\end{equation}
Applying Lemma \ref{lem: switching principle reflected currents gs class} gives
\begin{equation}
    \mathsf Z_{\Lambda\setminus S,\rho, \beta}^\emptyset[B_y\xleftrightarrow[]{\:\mathcal M_n\:}\mathbb H_n^{(N)}]\leq \mathcal \beta\sum_{y'\sim y} \mathsf Z^{\{y,\mathcal{R}_n(y')\}}_{\Lambda\setminus S,\beta}.
\end{equation}
Moreover, using Lemma \ref{lem: switching principle reflected currents gs class} one more time on the last line,
\begin{align*}
    \sum_{\substack{S\cap \mathbb H_n^{(N)}=\emptyset\\ S\supset B_0,B_x}}\mathsf Z^{\{0,x\}}_{\Lambda,\rho,\beta}[\mathcal{C}_n(0)=S]&=\mathsf Z^{\{0,x\}}_{\Lambda,\rho,\beta}[B_0,B_x\overset{\mathcal M_n}{\centernot\longleftrightarrow}\mathbb H_n^{(N)}]
    \\&\leq \mathsf Z^{\{0,x\}}_{\Lambda,\rho,\beta}-\mathsf Z^{\{0,x\}}_{\Lambda,\rho,\beta}[\sn \cap B_x\xleftrightarrow[]{\:\mathcal M_n\:}\mathbb H_n^{(N)}]
    \\
    &=\mathsf Z^{\{0,x\}}_{\Lambda,\rho,\beta}-\mathsf Z^{\{0,\mathcal{R}_n(x)\}}_{\Lambda,\rho,\beta}.
\end{align*}
Collecting the above work and taking the limit as $\Lambda\nearrow\mathbb Z^d$, we get the result.
\end{proof}

\subsection{Proof of the theorem for general models in the GS class}\label{section: extension to all models in the gs class}
We now turn to the extension of Theorem \ref{prop: intermediate step lower bound} to all models in the GS class (and in particular to the $\varphi^4$ model). It suffices to show that the inequality is stable under weak limits $\rho_k\rightarrow \rho$ for $(\rho_k)_{k\geq 1}$ a sequence of Ising-type measures in the GS class. Similar arguments had to be used in \cite{ADC,Pan23+} and we will essentially import the tools developed there. The extension of the result  follows readily from this proposition. We define $L'(\rho,\beta)$ as in \eqref{eq:def L'}.
\begin{Prop}[\hspace{1pt}{\cite[Proposition~8.7]{Pan23+}}]\label{prop: input limit phi4} Let $d\geq 3$. Let $\rho$ be a measure in the GS class. Let $(\rho_k)_{k\geq 1}$ be a sequence of measures of the Ising type in the GS class that converges weakly to $\rho$. Then,
\begin{enumerate}
    \item[$(i)$] $\liminf \beta_c(\rho_k)\geq \beta_c(\rho)$,
    \item[$(ii)$] for every $\beta>0$, $\liminf L'(\rho_k,\beta)\geq L(\rho,\beta)$,
    \item[$(iii)$] for every $\beta< \beta_c(\rho)$, for every $x,y\in \mathbb Z^d$,
    \begin{equation}
        \lim_{k\rightarrow \infty}\langle \tau_x\tau_y\rangle_{\rho_k,\beta}=\langle \tau_x\tau_y\rangle_{\rho,\beta}.
    \end{equation}
\end{enumerate}
\end{Prop}

\begin{proof}[Proof of Theorem \textup{\ref{prop: intermediate step lower bound}} for general measures in the GS class] Let $\rho$ be a general measure in the GS class. Let $(\rho_k)_{k\geq 1}$ be a sequence of measures of the Ising type in the GS class which converges weakly to $\rho$. Using Theorem \ref{prop: intermediate step lower bound} together with Remarks \ref{rem: 1/2 to 1/4} and \ref{rem: uniform bound in the GS class}, we get $c_0,N_0$ such that, for every $k\geq 1$, $N_0\leq n\leq L'(\rho_k,\beta)$, and $\beta\leq \beta_c(\rho_k)$ ,
\begin{equation}\label{eq: proof extension general 1}
	\beta\sum_{\substack{x,y\in \Lambda_n\\y\sim x}}\left(\langle \tau_0\tau_x\rangle_{\rho_k,\beta}-\langle \tau_0\tau_{\mathcal{R}_n(x)}\rangle_{\rho_k,\beta}\right)\langle \tau_y \tau_{\mathcal{R}_n(y)}\rangle_{\rho_k,\beta}\geq c_0.
\end{equation}
Now, let $\beta<\beta_c(\rho)$ and $n\leq L(\rho,\beta)$. Thanks to Proposition \ref{prop: input limit phi4}, for $k$ large enough, one has that \eqref{eq: proof extension general 1} holds for such choice of $\beta,n$. The proof for $\beta<\beta_c(\rho)$ follows readily from taking $k$ to infinity and using Proposition \ref{prop: input limit phi4} $(iii)$. We then extend the result to $\beta_c(\rho)$ by continuity, taking the limit $\beta\rightarrow \beta_c(\rho)$.
\end{proof}

\begin{Rem} Proposition \ref{prop: input limit phi4} is proved for dimensions $d\geq 4$ in \cite{Pan23+}. This limitation arises from the argument used in the proof of $(iii)$, as detailed in \cite[Lemma~8.10]{Pan23+}. Importing the notations from the same paper, one needs to show that
\begin{equation}\label{eq: ending rem}
	\limsup_{n\rightarrow\infty}\limsup_{k\rightarrow \infty}\mathbb P^{xy}_{\rho_k,\beta}[\mathsf{ZZGS}_k(x,y;\ell,n,\infty)]=0.
\end{equation}
To prove that, one can use a first moment method together with the infrared bound (see \cite[Corollary~6.12]{Pan23+}). However, this bound not being sharp in dimension $3$, it is not sufficient to close the argument in this setup. It is possible to solve this issue by noticing that for $\beta<\beta_c(\rho)$, the exponential decay of $\langle \tau_0\tau_u\rangle_{\rho_k,\beta}$ is uniform in $k$. Indeed, using \eqref{eq: near-critical upper bound} and the fact that $\limsup L(\rho_k,\beta)\leq L(\rho,\beta)$: for $k$ large enough, if $u\in \mathbb Z^d$,
\begin{equation}
    \langle \tau_0\tau_u\rangle_{\rho_k,\beta}\leq C\exp\left(-c\frac{|u|}{2L(\rho,\beta)}\right),
\end{equation}
where $c,C>0$. One can then plug this bound in the strategy of \cite{Pan23+} to show that \eqref{eq: ending rem} holds when $d=3$.
\end{Rem}

%
%
%
%
%

\begin{Acknowledgements} We thank Alexis Prévost for stimulating discussions. We thank Trishen S. Gunaratnam, Tiancheng He, Gordon Slade, and two anonymous referees for useful comments. This project has received funding from the Swiss National Science Foundation, the NCCR SwissMAP, and the European Research Council (ERC) under the European Union’s Horizon 2020 research and innovation programme (grant agreement No. 757296). HDC acknowledges the support from the Simons collaboration on localization of waves. 
\end{Acknowledgements}

\begin{Conflict} The authors have no relevant financial or non-financial interests to disclose.	
\end{Conflict}

\bibliographystyle{alpha}
\bibliography{lb}
\end{document}